\documentclass[fleqn,11pt]{article}

\usepackage{amsmath, amscd, amsfonts, amssymb, graphicx, color}
\usepackage{hyperref}
\usepackage{cite}
\usepackage{amsmath}
\usepackage{amsfonts}
\usepackage{epsfig}
\usepackage{amssymb}
\usepackage{amsthm}
\usepackage{epstopdf}
\usepackage{footnote}
\usepackage{newlfont}
\usepackage{color}
\newtheorem{theorem}{Theorem}[section]
\newtheorem{lemma}[theorem]{Lemma}
\newtheorem{proposition}[theorem]{Proposition}

\newtheorem{definition}[theorem]{Definition}

\begin{document}

 \title{On the distance from a weakly normal matrix polynomial to
 matrix polynomials with a prescribed multiple eigenvalue}

 \author{E. Kokabifar\thanks{Department of Mathematics,
 Faculty of Science, Yazd University, Yazd, Iran
 (e.kokabifar@stu.yazd.ac.ir, loghmani@yazd.ac.ir).},\,
 G.B. Loghmani\footnotemark[1]\,\,
 and P.J. Psarrakos\thanks{Department of Mathematics,
 National Technical University of Athens, Zografou Campus,
 15780 Athens, Greece (ppsarr@math.ntua.gr).}}

\maketitle

\vspace{-6mm}

\begin{abstract}
Consider an $n \times n$ matrix polynomial $P(\lambda)$. An upper
bound for a spectral norm distance from $P(\lambda)$ to the set of
$n \times n$ matrix polynomials that have a given scalar
$\mu\in\mathbb{C}$ as a multiple eigenvalue was recently obtained
by Papathanasiou and Psarrakos (2008). This paper concerns a
refinement of this result for the case of weakly normal matrix
polynomials. A modification method is implemented and its
efficiency is verified by an illustrative example.
\end{abstract}



\section{Introduction}

Let $A$ be an $n\times n$ complex matrix and $\mu$ be a complex
number, and denote by $\mathcal{M}_{\mu}$ the set of $n\times n$
complex matrices that have $\mu \in \mathbb{C}$ as a multiple
eigenvalue. Malyshev \cite{malyshev} obtained the following
formula for the spectral norm distance from $A$ to
$\mathcal{M}_{\mu}$:
\begin{equation*}
  \mathop {\min } \limits_{B \in \mathcal{M}_{\mu}}
  {\left\| {A - B} \right\|_2} =
  \mathop {\max}\limits_{\gamma \ge 0} {s_{2n - 1}}
  \left( {\left[ {\begin{array}{*{20}{c}}
  {A - \mu I}&{\gamma {I_n}}  \\  0&{A - \mu I}
  \end{array}} \right]} \right),
\end{equation*}
where $\|\cdot\|_2$ denotes the spectral matrix norm (i.e., that
norm subordinate to the euclidean vector norm) and $\,s_1(\cdot)
\geq s_2(\cdot) \geq s_3(\cdot) \geq \cdots \,$ are the singular
values of the corresponding matrix in nonincreasing order.
Malyshev's work can be considered as a theoretical solution to
Wilkinson's problem, that is, the calculation of the distance from
a matrix $A \in \mathbb{C}^{n \times n}$ that has all its
eigenvalues simple to the $n\times n$ matrices with multiple
eigenvalues. Wilkinson introduced this distance in
\cite{wilkinson}, and some bounds for it were computed by Ruhe
\cite{ruhe}, Wilkinson \cite{wil1,wil2,wil3,wil4} and Demmel
\cite{demmel1}.

However, in the non-generic case where $A$ is a normal matrix,
Malyshev's formula is not directly applicable. In 2004, Ikramov
and Nazari \cite{ikramovasli} showed this point and obtained an
extension of Malyshev's method for normal matrices. Moreover,
Malyshev's results were extended by Lippert \cite{lipert} and
Gracia \cite{gracia}; in particular, they computed a spectral norm
distance from $A$ to the set of matrices that have two prescribed
eigenvalues and studied a nearest matrix with the two desired
eigenvalues. Nazari and Rajabi \cite{nazarirajabi} refined the
method obtained by Lippert and Gracia for the case of normal
matrices.

In 2008, Papathanasiou and Psarrakos \cite{papa} introduced and
studied a spectral norm distance from a $n\times n$ matrix
polynomial $P(\lambda)$ to the set of $n\times n$ matrix
polynomials that have a scalar $\mu\in\mathbb{C}$ as a multiple
eigenvalue. In particular, generalizing Malyshev's methodology,
they computed lower and upper bounds for this distance,
constructing an associated perturbation of $P(\lambda)$ for the
upper bound. Motivated by the above, in this note, we study the
case of weakly normal matrix polynomials. In the next section, we
give some definitions and present briefly some of the results of
\cite{papanormal,papa}. We also give an example of a normal matrix
polynomial where the method described in \cite{papa} for the
computation of the upper bound is not directly applicable. In
Section \ref{s3}, we prove that the methodology of \cite{papa} for
the computation of the upper bound is indeed not directly
applicable to weakly normal matrix polynomials, and in Section
\ref{s4}, we obtain a modified procedure to improve the method.
The same numerical example is considered to illustrate the
validity of the proposed technique.

\section{Preliminaries}   \label{s2}

For $A_0,A_1,\dots,A_m\in\mathbb{C}^{n \times n}$, with $\det(A_m) \ne
0$, and a complex variable $\lambda$, we define the \textit{matrix
polynomial}
\begin{equation}  \label{plambda}
   P(\lambda ) = A_m \lambda ^m  + A_{m - 1} \lambda ^{m - 1}  + \cdots + A_1 \lambda  + A_0 .
\end{equation}
The study of matrix polynomials, especially with regard to their
spectral analysis, has received a great deal of attention and has
been used in several applications
\cite{glr,kacz,lanc,markus,time}. Standard references for the
theory of matrix polynomials are \cite{glr,markus}. Here, some
definitions of matrix polynomials are briefly reviewed.

If for a scalar $\lambda_0 \in \mathbb{C}$ and some nonzero vector
$x_0 \in {\mathbb{C}^{n}}$, it holds that $P(\lambda_0) x_0 = 0$,
then the scalar $\lambda_0$ is called an \textit{eigenvalue} of
$P(\lambda)$ and the vector $x_0$ is known as a \textit{(right)
eigenvector} of $P(\lambda)$ corresponding to $\lambda_0$. The
\textit{spectrum} of $P(\lambda)$, denoted by $\sigma(P)$, is the
set of all eigenvalues of $P(\lambda)$. Since the leading
matrix-coefficient $A_m$ is nonsingular, the spectrum $\sigma(P)$
contains at most $mn$ distinct finite elements. The multiplicity
of an eigenvalue $\lambda_0 \in \sigma(P)$ as a root of the scalar
polynomial $\det P(\lambda)$ is said to be the \textit{algebraic
multiplicity} of $\lambda_0$, and the dimension of the null space
of the (constant) matrix $P(\lambda_0)$ is known as the
\textit{geometric multiplicity} of $\lambda_0$. The algebraic
multiplicity of an eigenvalue is always greater than or equal to
its geometric multiplicity. An eigenvalue is called
\textit{semisimple} if its algebraic and geometric multiplicities
are equal; otherwise, it is known as \textit{defective}.

\begin{definition}  \textup{
Let $P(\lambda)$ be a matrix polynomial as in (\ref{plambda}). If
there exists a unitary matrix $U \in \mathbb{C}^{n \times n}$ such
that $U^*P(\lambda)U$ is a diagonal matrix polynomial, then
$P(\lambda)$ is said to be \textit{weakly normal}. If, in
addition, all the eigenvalues of $P(\lambda)$ are semisimple, then
$P(\lambda)$ is called \textit{normal}.  }
\end{definition}

The suggested references on weakly normal and normal matrix
polynomials, and their properties are \cite{nn,papanormal}. Some
of the results of \cite{papanormal} are summarized in the next
proposition.

\begin{proposition} \textup{\cite{papanormal}} \label{weak}\
Let $P(\lambda) = A_m \lambda^m + \cdots + A_1 \lambda + A_0$ be a
matrix polynomial as in (\ref{plambda}). Then $P(\lambda)$ is
weakly normal if and only if one of the following (equivalent)
conditions holds.
\begin{description}
 \item[(i)  ] For every $\mu \in \mathbb{C}$, the matrix $P(\mu)$ is
              normal.
 \item[(ii) ] $A_0,A_1,\dots,A_m$ are normal and mutually commuting
              (i.e., $A_i A_j = A_j A_i$ for $i\ne j$).
 \item[(iii)] All the linear combinations of $A_0,A_1,\dots,A_m$ are normal matrices.
 \item[(iv) ] There exists a unitary matrix $U\in\mathbb{C}^{n \times n}$
              such that $U^* A_j U$ is diagonal for every $j=0,1,\dots,m$.
\end{description}
\end{proposition}

As mentioned, Papathanasiou and Psarrakos \cite{papa} introduced a
spectral norm distance from a matrix polynomial $P(\lambda)$ to
the matrix polynomials that have $\mu$ as a multiple eigenvalue,
and computed lower and upper bounds for this distance. Consider
(additive) perturbations of $P(\lambda)$ of the form
\begin{equation} \label{eq:polyQ}
  Q(\lambda) = P(\lambda) + \Delta (\lambda)
  = (A_m+\Delta_m) \lambda ^m  + \cdots + (A_1+\Delta_1) \lambda  + A_0 + \Delta_0 ,
\end{equation}
where the matrices $\Delta_0 , \Delta_1 , \dots , \Delta_m \in \mathbb{C}^{n
\times n}$ are arbitrary. For a given parameter $\epsilon >0$ and a
given set of nonnegative weights $\textup{w} = \{ w_0, w_1, \dots , w_m
\}$ with $w_0>0$, define the class of admissible perturbed
matrix polynomials
\[
    \mathcal{B}(P,\epsilon,\textup{w}) =
    \left \{ Q(\lambda) \; \mbox{as in} \;(\ref{eq:polyQ}) :
    \| \Delta_j \|_2 \le \epsilon \, w_j ,\, j=0,1,\dots ,m \right \}
\]
and the scalar polynomial
$w(\lambda) = w_m \lambda^m + w_{m-1} \lambda^{m-1} + \cdots +
w_1 \lambda + w_0$. Note that the weights $w_0,w_1,\dots,w_m$ allow
freedom in how perturbations are measured.

For any real number $\gamma \in [ 0 , + \infty )$, we define the $2n\times 2n$
matrix polynomial
\[
 F\left[ {P(\lambda );\gamma } \right] = {\left[ {\begin{array}{*{20}{c}}
 {P(\lambda )}&0 \\  {\gamma P'(\lambda )}&{P(\lambda )} \end{array}} \right ]} ,
\]
where $P'(\lambda )$ denotes the derivative of $P(\lambda )$ with
respect to $\lambda$.

\begin{lemma}  \textup{\cite[Lemma 17]{papa}} \label{dotaee}
Let $\mu \in \mathbb{C}$ and $\gamma_*> 0$ be a point where the singular value $s_{2n -
1} (F[P(\mu );\gamma])$ attains its maximum value, and denote $s_* = s_{2n -
1} (F[P(\mu);\gamma_*])>0$. Then there exists a pair $\left[
{\begin{array}{*{20}c}
   {u_1 (\gamma _* )}  \\
   {u_2 (\gamma _* )}  \\
\end{array}} \right], \left[ {\begin{array}{*{20}c}
   {v_1 (\gamma _* )}  \\
   {v_2 (\gamma _* )}  \\
\end{array}} \right] \in\mathbb{C} ^{2n}~( u_k (\gamma _* ),v_k (\gamma _* ) \in \mathbb{C}^n,~k =
1,2)$ of left and right singular vectors of $F[P(\mu);\gamma_*]$ corresponding to $s_*$, respectively,
such that
 \begin{description}
 \item[(1)] $u^* _2 (\gamma_*) P'(\mu) v_1 (\gamma_* ) = 0$, and
 \item[(2)] the $\,n\times 2\,$ matrices $\, U(\gamma_* ) = \left [ u_1 (\gamma_*)~u_2 (\gamma_*) \right]\,$ and
           $\,V(\gamma_* ) = \left[ v_1(\gamma_* ) ~ v_2 (\gamma_* ) \right]\,$ satisfy
           $\,U^*(\gamma_* ) U(\gamma_* ) = V^* (\gamma_* ) V(\gamma_* )$.
 \end{description}
Moreover, it is remarkable that (1) implies (2) (see the proof of
Lemma 17 in \cite{papa}).
\end{lemma}

Consider the quantity
$\phi = \frac{{w'(\left| \mu \right|)}}{{w(\left| \mu \right|)}}
\frac{{\bar \mu }}{{\left| \mu  \right|}}$,
where, by convention, we set $\frac{{\bar \mu }}{{\left| \mu  \right|}}=0$
whenever $\mu=0$. Let also ${V}({\gamma_*})^\dag$ be the \emph{Moore-Penrose pseudoinverse}
of ${V}({\gamma_*})$. For the pair of singular vectors $\left[
{\begin{array}{*{20}c}
   {u_1 (\gamma _* )}  \\
   {u_2 (\gamma _* )}  \\
\end{array}} \right], \left[ {\begin{array}{*{20}c}
   {v_1 (\gamma _* )}  \\
   {v_2 (\gamma _* )}  \\
\end{array}} \right] \in\mathbb{C} ^{2n}$ of Lemma \ref{dotaee},
define the $n\times n$ matrix
\begin{equation*}
  {\Delta _{{\gamma _*}}} =  - {s_*}U({\gamma _*})
   \left[ {\begin{array}{*{20}{c}}
   1  &  {-{\gamma_*} \phi } \\  0 & 1
   \end{array}} \right]V{({\gamma _*})^\dag } .
\end{equation*}

\begin{theorem} \textup{\cite[Theorem 19]{papa}} \label{amultistar}
Let $P(\lambda)$ be a matrix polynomial as in (\ref{plambda}), and let
$\textup{w} = \{ w_0, w_1, \dots , w_m \}$, with $w_0>0$, be a set of
nonnegative weights. Suppose that $\mu\in\mathbb{C} \backslash \sigma(P')$,
$\gamma_*> 0$ is a point where the singular value $s_{2n - 1} (F[P(\mu );\gamma])$
attains its maximum value, and $s_* = s_{2n - 1} (F[P(\mu);\gamma_*])>0$.
Then, for the pair of singular vectors $\left [  \begin{array}{c}  u_1(\gamma_*)
\\ u_2(\gamma_*) \end{array} \right ] , \left[ \begin{array}{c} v_1(\gamma_*) \\
v_2(\gamma_*)  \end{array} \right ] \in \mathbb{C}^{2n}$
of Lemma \ref{dotaee}, we have
\begin{eqnarray*}
   & &  \min \left \{ \epsilon \geq 0 : \exists \;
      Q(\lambda) \in \mathcal{B}(P,\epsilon,\textup{w})
      \;\mbox{with $\mu$ as a multiple eigenvalue} \right \}  \\
   & &  \le \,
      \frac{s_*}{w(|\mu|)} \left \| V(\gamma_*) \left [
      \begin{array}{cc}
           1    &  -\gamma_*\,\phi  \\
           0    &   1   \\
      \end{array} \right ]
      V(\gamma_*)^{\dagger} \right \|  .
\end{eqnarray*}
Moreover, the perturbed matrix polynomial
\begin{equation} \label{eq:pertQgastar}
         Q_{\gamma_*}(\lambda)
      =  P(\lambda) +  \Delta_{\gamma_{*}}(\lambda)
      =  P(\lambda) + \sum_{j=0}^m  \frac{w_j}{w(|\mu|)}
         \left ( \frac{\overline{\mu}}{|\mu|} \right )^j
         \Delta_{\gamma_{*}} \, \lambda^j ,
\end{equation}
lies on the boundary of the set $\,\mathcal{B} \left( P ,
 \frac{s_*}{w(|\mu|)} \left \| V(\gamma_*) \left [
      \begin{array}{cc}
           1    &  -\gamma_*\,\phi  \\
           0    &   1   \\
      \end{array} \right ]
      V(\gamma_*)^{\dagger} \right \| , \textup{w} \right)$
and has $\mu$ as a (multiple) defective eigenvalue.
\end{theorem}

Some numerical examples in Section 8 of \cite{papa} illustrate the effectiveness of
the upper bound of Theorem \ref{amultistar}. In all these examples, $s_*$ is a
simple singular value, and consequently, the singular vectors $\left [  \begin{array}{c}
u_1(\gamma_*) \\ u_2(\gamma_*) \end{array} \right ] , \left[ \begin{array}{c} v_1(\gamma_*) \\
v_2(\gamma_*)  \end{array} \right ] \in \mathbb{C}^{2n}$ of Lemma
\ref{dotaee} are directly computable (due to their essential
uniqueness). Let us now consider the normal (in particular,
diagonal) matrix polynomial
\begin{equation}\label{pnormal}
 P(\lambda ) = \left[ {\begin{array}{*{20}{c}}
 1&0&0\\
 0&1&0\\
 0&0&1
 \end{array}} \right]{\lambda ^2} + \left[ {\begin{array}{*{20}{c}}
 { - 3}&0&0\\
 0&{ - 1}&0\\
 0&0&3
 \end{array}} \right]\lambda  + \left[ {\begin{array}{*{20}{c}}
 2&0&0\\
 0&0&0\\
 0&0&2
 \end{array}} \right]
\end{equation}
that is borrowed from \cite[Section 3]{papanormal}. Let also the
set of weights $w = \left\{ {1,1,1} \right\}$ and the scalar
$\mu=-4$. The singular value $s_{5}(F[P(-4);\gamma])$ attains its
maximum value at $\gamma_{*}=2.0180$, and at this point, we have
$s_* = s_{5}(F[P(-4);2.0180]) = s_{4}(F[P(-4);2.0180]) = 12.8841$;
i.e., $s_*$ is a multiple singular value of matrix
$F[P(-4);2.0180]$. A left and a right singular vectors of
$F[P(-4);2.0180]$ corresponding to $s_*$ are
\[
   \left[ {\begin{array}{*{20}c}
   {u_1 (\gamma _* )}  \\
   {u_2 (\gamma _* )}  \\
   \end{array}} \right] = \left[  \begin{array}{c}
    0  \\  0.8407  \\  0  \\  0  \\  0.5416 \\ 0
   \end{array} \right]  \;\;\; \mbox{and} \;\;\;
   \left[ {\begin{array}{*{20}c}
   {v_1 (\gamma _* )}  \\
   {v_2 (\gamma _* )}  \\
   \end{array}} \right] = \left[  \begin{array}{c}
    0  \\  0.5416  \\  0  \\  0  \\  0.8407 \\ 0
   \end{array} \right] ,
\]
respectively, and they yield the perturbed matrix polynomial (see
(\ref{eq:pertQgastar}))
\begin{eqnarray*}
 {Q_{{\gamma _*}}}(\lambda ) = \left[ {\begin{array}{*{20}{c}}
 1&0&0\\
 0&{0.0664}&0\\
 0&0&1
 \end{array}} \right]{\lambda ^2} + \left[ {\begin{array}{*{20}{c}}
 { - 3}&0&0\\
 0&{ - 0.0664}&0\\
 0&0&3
 \end{array}} \right]\lambda  +\left[ {\begin{array}{*{20}{c}}
 2&0&0\\
 0&{ - 0.9336}&0\\
 0&0&2
 \end{array}} \right] .
\end{eqnarray*}
One can see that $\mu=-4$ is not a multiple eigenvalue of
$Q_{{\gamma _*}}(\lambda )$. Moreover, properties (1) and (2) of
Lemma \ref{dotaee} do not hold since $u_2^*({\gamma_*}) P'(\mu)
v_1({\gamma_*}) = -2.6396 \ne 0$ and $\left\| U^* ({\gamma_*})
U({\gamma_*}) - V^*({\gamma_*}) V({\gamma_*}) \right\|_2 = 0.4134
\ne 0$.

Clearly, this example verifies that the computation of appropriate
singular vectors which satisfy (1) and (2) of Lemma \ref{dotaee} is
still an open problem when $s_*$ is a multiple singular value. In
the next section, we obtain that for weakly normal matrix polynomials,
$s_*$ is always a multiple singular value, and in Section \ref{s4}, we solve the problem
of calculation of the desired singular vectors of Lemma \ref{dotaee}.

\section{Weakly normal matrix polynomials}  \label{s3}

In this section, by extending the analysis performed in
\cite{onaremarkable}, we prove that $s_*$ is always a multiple
singular value of $F[P(\mu);\gamma_*]$ when $P(\lambda)$ is a
weakly normal matrix polynomial.

Let $P(\lambda)$ be a weakly normal matrix polynomial, and let
$\mu\in\mathbb{C} \backslash \sigma(P')$. By Proposition
\ref{weak}\,(iv), it follows that there exists a unitary matrix
$U\in\mathbb{C}^{n\times n}$ such that all matrices $U^*A_0 U,
U^*A_1 U,\dots, U^*A_m U$ are diagonal. Hence, $U^*P(\mu)U$ and
$U^*P(\mu)'U$ are also diagonal matrices; in particular,
\[
   U^* P(\mu) U =   \textup{diag} \{ \zeta_1 , \zeta_2 , \dots , \zeta_n \}
   \;\;\; \mbox{and} \;\;\;
   U^* P(\mu)' U =  \textup{diag} \{ \xi_1 , \xi_2 , \dots , \xi_n \}  ,
\]
where all scalars $\xi_1,\xi_2,\dots,\xi_n \in\mathbb{C}$ are
\textit{nonzero} (recall that $P'(\mu)$ is nonsingular) and,
without loss of generality, we assume that
\[
    \left| \zeta_1 \right| \geq \left| \zeta_2 \right| \geq \cdots \geq \left| \zeta_n \right| .
\]
As a consequence,
\begin{eqnarray*}
   \left[ {\begin{array}{*{20}{c}}  U^* & 0 \\ 0 & U^* \end{array}} \right]
   F[P(\mu);\gamma ]
   \left[ {\begin{array}{*{20}{c}}  U & 0 \\ 0 & U \end{array}} \right]
   &=& \left[ {\begin{array}{*{20}{c}}  U^* & 0 \\ 0 & U^* \end{array}} \right]
       \left[ {\begin{array}{*{20}{c}}
       {P(\mu )} & 0 \\ {\gamma P'(\mu)} & {P(\mu)} \end{array}} \right]
       \left[ {\begin{array}{*{20}{c}}  U & 0 \\ 0 & U \end{array}} \right] \\
   &=& \left[ {\begin{array}{*{20}{c}}
       \textup{diag} \{ \zeta_1 , \zeta_2 , \dots , \zeta_n \} & 0 \\
       {\gamma\, \textup{diag} \{ \xi_1 , \xi_2 , \dots , \xi_n \}} &
       \textup{diag} \{ \zeta_1 , \zeta_2 , \dots , \zeta_n \} \end{array}} \right] .
\end{eqnarray*}

It is straightforward to verify that there is a $2n \times 2n$ permutation matrix $R$ such that
\begin{eqnarray*}
  & &   R  \left[ {\begin{array}{*{20}{c}}
        \textup{diag} \{ \zeta_1 , \zeta_2 , \dots , \zeta_n \} & 0 \\
        {\gamma\, \textup{diag} \{ \xi_1 , \xi_2 , \dots , \xi_n \}} &
        \textup{diag} \{ \zeta_1 , \zeta_2 , \dots , \zeta_n \} \end{array}} \right]
        R^T               \\
  & &   = \left[ \begin{array}{*{20}c} \zeta_1 & 0 \\ \gamma\,\xi_1 & \zeta_1 \end{array} \right]
        \oplus \left[ \begin{array}{*{20}c} \zeta_2 & 0 \\ \gamma\,\xi_2  & \zeta_2 \end{array} \right]
        \oplus \, \cdots \, \oplus
        \left[ \begin{array}{*{20}c} \zeta_n & 0 \\ \gamma\,\xi_n  & \zeta_n\end{array} \right]  .
\end{eqnarray*}
The fact that singular values of a matrix are invariant under
unitary similarity implies that the $2n \times 2n$ matrices
\[
  F[P(\mu);\gamma ]  \;\;\;\; \mbox{and} \;\;\;\;
  \left[ \begin{array}{*{20}c} \zeta_1 & 0 \\ \gamma\,\xi_1 & \zeta_1 \end{array} \right]
  \oplus \left[ \begin{array}{*{20}c} \zeta_2 & 0 \\ \gamma\,\xi_2  & \zeta_2 \end{array} \right]
  \oplus \, \cdots \, \oplus
  \left[ \begin{array}{*{20}c} \zeta_n & 0 \\ \gamma\,\xi_n  & \zeta_n\end{array} \right]
\]
have the same singular values.  Therefore, in what follows, we are
focused on the singular values of $\left[ \begin{array}{*{20}c}
\zeta_1 & 0 \\ \gamma\,\xi_1 & \zeta_1 \end{array} \right] \oplus
\left[ \begin{array}{*{20}c} \zeta_2 & 0 \\ \gamma\,\xi_2  &
\zeta_2 \end{array} \right] \oplus \cdots \oplus \left[
\begin{array}{*{20}c} \zeta_n & 0 \\ \gamma\,\xi_n  &
\zeta_n\end{array} \right]$, which are the union of the singular
values of $\,\left[ {\begin{array}{*{20}c} {\zeta_i } & 0  \\
\gamma\,\xi_i  & \zeta_i \end{array}} \right]$, $\,i=1,2,\dots,n$.

For any $i=1,2,\dots,n$, let $s_{i,1}(\gamma)\geq s_{i,2}(\gamma)$ be the singular values of
$\left[ {\begin{array}{*{20}c} {\zeta_i } & 0  \\ \gamma\,\xi_i  & \zeta_i \end{array}} \right]$,
and consider the characteristic polynomial of matrix
\[
 \left[ {\begin{array}{*{20}c} {\zeta_i } & 0  \\
 \gamma\,\xi_i  & \zeta_i \end{array}} \right]^* \left[ {\begin{array}{*{20}c} {\zeta_i } & 0  \\
 \gamma\,\xi_i  & \zeta_i \end{array}} \right] = \left[ {\begin{array}{*{20}c}
   \left| \zeta_i \right|^2 + \gamma^2 \left| \xi_i \right|^2
   & \gamma\,\overline{\xi}_i \,\zeta_i   \\ \gamma \, \xi_i \,\overline{\zeta}_i
   & \left | \zeta_i \right |^2 \end{array}} \right] ,
\]
that is,
\[
 \det \left( t I - \left[ {\begin{array}{*{20}c}
   \left| \zeta_i \right|^2 + \gamma^2 \left| \xi_i \right|^2
   & \gamma \, \overline{\xi}_i \,\zeta_i   \\ \gamma \, \xi_i \,\overline{\zeta}_i
   & \left | \zeta_i \right |^2 \end{array}} \right] \right)
 = \, t^2 - \left( 2 \left| \zeta_i \right|^2 + \gamma^2 \left| \xi_i \right|^2 \right) t
   + \left| \zeta_i \right|^4 .
\]
The positive square roots of the eigenvalues of matrix $\,\left[ {\begin{array}{*{20}c}
   \left| \zeta_i \right|^2 + \gamma \left| \xi_i \right|^2
   & \gamma \, \overline{\xi}_i \,\zeta_i   \\ \gamma \, \xi_i \,\overline{\zeta}_i
   & \left | \zeta_i \right |^2 \end{array}} \right]\,$
are the singular values of matrix $\left[ {\begin{array}{*{20}c} {\zeta_i } & 0  \\
\gamma\,\xi_i  & \zeta_i \end{array}} \right]$, namely,
\[
  s_{i,1}(\gamma) = \sqrt {\left| {\zeta_i} \right|^2 + \frac{{\gamma^2\left| \xi_i \right|^2}}{2}
  + \gamma \left| \xi_i \right| \sqrt {\left| {\zeta_i} \right|^2
  + \frac{{\gamma^2\left| \xi_i \right|^2}}{4}}},
\]
and
\[
  s_{i,2}(\gamma) = \sqrt{\left| {\zeta_i} \right|^2 + \frac{{\gamma^2\left| \xi_i \right|^2}}{2}
  - \gamma \left| \xi_i \right| \sqrt{\left| {\zeta_i} \right|^2 + \frac{{\gamma^2\left| \xi_i \right|^2}}{4}} } \, .
\]
As $\gamma\ge 0$ increases, $s_{i,1}(\gamma)$ increases and
$\lim\limits_{\gamma\rightarrow +\infty} s_{i,1}(\gamma) =
+\infty$, while $s_{i,2}(\gamma)$ decreases and
$\lim\limits_{\gamma\rightarrow +\infty} s_{i,2}(\gamma) = 0$
(recall that $\left|\xi_i\right| > 0$, $i=1,2,\dots,n$). Also, it
is apparent that
\[
 s_{i,2}(\gamma) \le \left| \zeta_i \right| \le s_{i,1}(\gamma)
 \;\;\;\mbox{and}\;\;\;
 s_{i,1}(0)=s_{i,2}(0) = \left| \zeta_i \right| .
\]
Next we consider two cases with respect to $\left| \zeta_{n-1} \right|$ and $\left| \zeta_n \right|$.

\textit{Case 1.}\ Suppose $\left| \zeta_n \right| < \left|
\zeta_{n-1} \right|$.
At $\gamma = 0$, it holds that $s_{n,1}(0) = \left| \zeta_n
\right| < \left| \zeta_{n-1} \right| = s_{n-1,2}(0)$. According to
the above discussion, as the nonnegative variable $\gamma$
increases from zero, the functions
\[
   s_{1,1}(\gamma),\; s_{2,1}(\gamma),\; \dots ,\; s_{n-1,1}(\gamma),\; s_{n,1}(\gamma)
\]
increase to $+\infty$, whereas the functions
\[
   s_{1,2}(\gamma),\; s_{2,2}(\gamma),\; \dots ,\; s_{n-1,2}(\gamma),\; s_{n,2}(\gamma)
\]
decrease to $\,0$. Let $(\gamma_0,s_0)$ be the first point in
$\mathbb{R}^2$ where the graph of the increasing function
$s_{n,1}(\gamma)$ intersects the graph of one of the $n-1$
decreasing functions $s_{1,2}(\gamma) , s_{2,2}(\gamma) , \dots ,
s_{n-1,2}(\gamma)$, say $\,s_{\kappa,2}(\gamma)\,$ (for some
$\kappa \in \{ 1,2,\dots,n-1\}$). Note that by the definition of
$s_{i,1}(\gamma)$ and $s_{i,2}(\gamma)$ ($i=1,2,\dots,n$), $s_0$
lies in the open interval $( 0 , \left| \zeta_{n-1} \right| )$ and
the graph of $s_{n,1}(\gamma)$ cannot intersect the graph of one
of the increasing functions $s_{1,1}(\gamma), s_{2,1}(\gamma),
\dots, s_{n-1,1}(\gamma)$ for $\gamma \le \gamma_0$.

Since $s_{n,2}(\gamma)$ and $s_{\kappa,2}(\gamma)$ are both
decreasing functions in $\gamma \geq 0$, it follows that (see Fig.
1 below, where $\kappa = n-1 = 2$)
\[
     \gamma_* = \gamma_0   \;\;\;\mbox{and}\;\;\;
     s_* = s_0 = s_{2n-1}(F[P(\mu);\gamma_*]) = s_{n,1}(\gamma_*)
         = s_{\kappa,2}(\gamma_*) = s_{2n-2} (F[P(\mu);\gamma_*]) .
\]
Hence, when $\left| \zeta_n \right| < \left| \zeta_{n-1} \right|$,
$\,\gamma_*\,$ is the minimum positive root of one of the
equations
\[
   s_{n,1} (\gamma) = s_{n-1,2} (\gamma) , \;\;
   s_{n,1} (\gamma) = s_{n-2,2} (\gamma) , \;\;
   \dots , \;\;
   s_{n,1} (\gamma) = s_{1,2} (\gamma)
\]
and $\,s_*\,$ is a multiple singular value of
$F[P(\mu);\gamma_*]$.

\textit{Case 2.}\ Suppose $\left| \zeta_n \right| = \left|
\zeta_{n-1} \right|$.
Then, it follows that $s_{n,1}(\gamma) = s_{n-1,1}(\gamma)$ and $s_{n,2}(\gamma) = s_{n-1,2}(\gamma)$.
Moreover, one can see that at $\gamma=0$,
\[
  s_{n,1}(0) = s_{n,2}(0) = s_{n-1,1}(0) = s_{n-1,2}(0)
  = \left| \zeta_n \right| = \left| \zeta_{n-1} \right|,
\]
i.e.,
\[
  s_{2n} (F[P(\mu);0]) = s_{2n-1} (F[P(\mu);0]) =  s_{2n-2} (F[P(\mu);0]) =  s_{2n-3} (F[P(\mu);0])
  = \left| \zeta_n \right| = \left| \zeta_{n-1} \right| .
\]
Since $s_{n,2}(\gamma)$ and $s_{n-1,2}(\gamma))$ are decreasing
functions in $\gamma \geq 0$, $\,s_{2n-1} (F[P(\mu);\gamma])$
attains its maximum value $s_*$ at $\gamma = 0 = \gamma_*$, and
$s_*$ is a multiple singular value of $F[P(\mu);0]$. In this
non-generic case, an upper bound and an associate perturbed matrix
polynomial can be computed by the method described in Section 6 of
\cite{papa}.

Hence, we have the following result.

\begin{theorem} \label{result}
Let $P(\lambda)$ in (\ref{plambda}) be a weakly normal matrix polynomial, and let
$\mu\in\mathbb{C} \backslash \sigma(P')$. If $\gamma_*> 0$ is a point where the
singular value $s_{2n - 1} (F[P(\mu );\gamma])$ attains its maximum value, then
$s_* = s_{2n - 1} (F[P(\mu);\gamma_*])>0$ is a multiple singular value of $F[P(\mu);\gamma_*]$.
\end{theorem}

\section{Computing the desired singular vectors}  \label{s4}

In this section, we apply a technique proposed in
\cite{ikramovasli} (see also the proof of Lemma 5 in
\cite{malyshev}) to compute suitable singular vectors of
$F[P(\mu);\gamma_*]$ corresponding to the singular value $s_*$,
which satisfy (1) and (2) of Lemma \ref{dotaee}. It is remarkable
that the proposed technique can be applied to \textit{general}
matrix polynomials and not only to weakly normal matrix
polynomials.

\subsection{The case of multiplicity 2}

First we consider the case where $\gamma_* > 0$ and the multiplicity
of the singular value $s_* > 0$ is equal to $2$, and we work on the
example of Section \ref{s2}.

Recall that for the normal matrix polynomial $P(\lambda)$ in
(\ref{pnormal}) and for $\mu=-4$, the singular value
$s_{2n-1}(F[P(\mu);\gamma]) = s_{5}(F[P(-4);\gamma])$ attains its
maximum value at $\gamma_{*}=2.0180$ and $s_* =
s_{5}(F[P(-4);2.0180]) = s_{4}(F[P(-4);2.0180]) = 12.8841$ (i.e.,
$s_*$ is a double singular value of $F[P(-4);2.0180]$). Two pairs
of left and a right singular vectors of $F[P(-4);2.0180]$
corresponding to $s_*$, which do not satisfy properties (1) and
(2) of Lemma \ref{dotaee} are
\[
   \left[ {\begin{array}{*{20}c}
   {u_1 (\gamma _* )}  \\
   {u_2 (\gamma _* )}  \\
   \end{array}} \right] = \left[  \begin{array}{c}
    0  \\  0.8407  \\  0  \\  0  \\  0.5416 \\ 0
   \end{array} \right] , \;\;\;
   \left[ {\begin{array}{*{20}c}
   {v_1 (\gamma _* )}  \\
   {v_2 (\gamma _* )}  \\
   \end{array}} \right] = \left[  \begin{array}{c}
    0  \\  0.5416  \\  0  \\  0  \\  0.8407 \\ 0
   \end{array} \right],
\]
and
\[
   \left[ {\begin{array}{*{20}c}
   {\hat{u}_1 (\gamma _* )}  \\
   {\hat{u}_2 (\gamma _* )}  \\
   \end{array}} \right] = \left[  \begin{array}{c}
    0  \\  0  \\  -0.4222  \\  0  \\  0  \\ 0.9065
   \end{array} \right]  , \;\;\;
   \left[ {\begin{array}{*{20}c}
   {\hat{v}_1 (\gamma _* )}  \\
   {\hat{v}_2 (\gamma _* )}  \\
   \end{array}} \right] = \left[  \begin{array}{c}
    0  \\  0  \\  -0.9065  \\  0  \\  0  \\ 0.4222
   \end{array} \right] .
\]
In particular, we have
\[
  u_2({\gamma_*})^* P'(-4) v_1({\gamma_*})
  = - 2.6396 \ne 0 \;\;\; \mbox{and} \;\;\;
  \hat{u}_2({\gamma _*})^* P'(-4)
 \hat{v}_1({\gamma _*}) = 4.1089 \ne 0 .
\]

\begin{figure}
\centering
\includegraphics[width=0.60\linewidth]{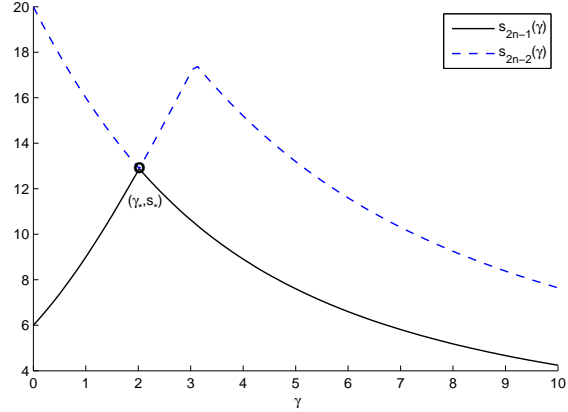}
\caption{\small The singular values $s_{2n-1} (F[P(\mu);\gamma ])$
(solid line) and $s_{2n-2} (F[P(\mu);\gamma])$ (dashed line).}
\label{fig:fig11}
\end{figure}
In Figure \ref{fig:fig11}, the graphs of
\[
   s_{2n-1} (F[P(\mu);\gamma ]) = s_5 (F[P(-4);\gamma ])
  \;\;\;\, \mbox{and} \;\;\;\,
   s_{2n-2} (F[P(\mu);\gamma ]) = s_4 (F[P(-4);\gamma ])
\]
are plotted for $\gamma \in \left[ {0,10} \right]$, and their
common point $(\gamma_*, s_*)=(2.0180,12.8841)$ is marked with
``$\circ$''. With respect to the discussion in the previous
section, it is worth noting that in this example, the graph of
$s_{2,2} (\gamma)$ (that is, $s_{n-1,2}(\gamma)$) is the graph of
the decreasing functions $s_{1,2} (\gamma)$ and $s_{2,2} (\gamma)$
that intersects first the graph of the increasing function
$s_{3,1} (\gamma)$ (that is, $s_{n,1}(\gamma)$). Moreover, it is
apparent that $s_{2n-1} (F[P(\mu);\gamma ])$ and $s_{2n-2}
(F[P(\mu);\gamma ])$ are non-differentiable functions at
$\gamma_*$.

Since $s_* = s_5 (F[P(-4);2.0180]) = s_4 (F[P(-4);2.0180]) =
12.8841$ is a double singular value, the pairs of unit vectors
 $\left[ {\begin{array}{*{20}c}
   u_1 (\gamma _* )  \\
   u_2 (\gamma _* )  \\
   \end{array}} \right],
 \left[ {\begin{array}{*{20}c}
   {\hat{u}_1 (\gamma _* )}  \\
   {\hat{u}_2 (\gamma _* )}
   \end{array}} \right]$ and
 $\left[ {\begin{array}{*{20}c}
   v_1 (\gamma _* )  \\
   v_2 (\gamma _* )  \\
   \end{array}} \right],
 \left[ {\begin{array}{*{20}c}
   {\hat{v}_1 (\gamma _* )}  \\
   {\hat{v}_2 (\gamma _* )}
   \end{array}} \right]$
form orthonormal bases of the left and right singular subspaces
corresponding to $s_*$, respectively. So, recalling that in Lemma
\ref{dotaee}, assertion (1) yields assertion (2), henceforth we
are looking for a pair of unit vectors
\begin{equation}  \label{uvnew}
   \left[ {\begin{array}{*{20}c}
   {\tilde{u}_1 (\gamma_* )}  \\
   {\tilde{u}_2 (\gamma_* )}  \end{array}} \right]
   = \alpha \left[ {\begin{array}{*{20}c}
    u_1 (\gamma_*)  \\ u_2 (\gamma_*)  \end{array}} \right]
   + \beta \left[ {\begin{array}{*{20}c}
   \hat{u}_1 (\gamma_*)  \\ \hat{u}_2 (\gamma_*)  \end{array}} \right]
   ,\;\;
   \left[ {\begin{array}{*{20}c}
   \tilde{v}_1 (\gamma_*) \\
   \tilde{v}_2 (\gamma_*) \end{array}} \right]
   = \alpha \left[ {\begin{array}{*{20}c}
   v_1 (\gamma_*)  \\ v_2 (\gamma_*) \end{array}} \right]
   + \beta \left[ {\begin{array}{*{20}c}
   \hat{v}_1 (\gamma_*)  \\ \hat{v}_2 (\gamma_*) \end{array}} \right]
\end{equation}
such that
\begin{equation}\label{upvnew}
  \tilde{u}_2({\gamma_*})^*P'(\mu)\tilde{v}_1({\gamma_*}) = 0 ,
\end{equation}
where the scalars $\alpha,\beta\in\mathbb{C}$ satisfy $\left|
\alpha \right|^2 + \left| \beta  \right|^2 = 1$. By substituting
the unknown singular vectors of (\ref{uvnew}) into (\ref{upvnew}),
we obtain
\begin{equation}\label{alphabeta}
 \left[ {\begin{array}{*{20}{c}}
 \overline{\alpha}& \overline{\beta}
 \end{array}} \right] \,M\,
 \left[ {\begin{array}{*{20}{c}}
 \alpha \\
 \beta
 \end{array}} \right] = 0,
\end{equation}
where
\begin{equation} \label{matrixM}
M = \left[ {\begin{array}{*{20}{c}}
  u_2 ({\gamma_*})^* P'(\mu) v_1 ({\gamma _*})     & u_2 ({\gamma_*})^* P'(\mu) \hat{v}_1 ({\gamma _*})   \\
  \hat{u}_2 ({\gamma_*})^* P'(\mu) v_1({\gamma_*}) & \hat{u}_2 ({\gamma _*})^* P'(\mu) \hat{v}_1 ({\gamma _*})
 \end{array}  } \right].
\end{equation}

\begin{lemma}  \label{hermitianM}
The matrix $M$ in (\ref{matrixM}) is always hermitian.
\end{lemma}

\begin{proof} Recall that $\gamma_*$ and $s_*$ are positive.
By the proof of Lemma 17 in \cite{papa}, it follows that the
diagonal entries of matrix $M$ are real.

By the definition of the pairs of singular vectors
\[
   \left[ {\begin{array}{*{20}c}
   {u_1 (\gamma _* )}  \\
   {u_2 (\gamma _* )}  \\
   \end{array}} \right] ,
   \left[ {\begin{array}{*{20}c}
   {v_1 (\gamma _* )}  \\
   {v_2 (\gamma _* )}  \\
   \end{array}} \right]
   \;\;\; \mbox{and}  \;\;\;
   \left[ {\begin{array}{*{20}c}
   {\hat{u}_1 (\gamma _* )}  \\
   {\hat{u}_2 (\gamma _* )}  \\
   \end{array}} \right] ,
   \left[ {\begin{array}{*{20}c}
   {\hat{v}_1 (\gamma _* )}  \\
   {\hat{v}_2 (\gamma _* )}  \\
   \end{array}} \right]
\]
of $F[P(\mu);\gamma_*]$ corresponding to $s_*$, we have
\[
\left\{ {\begin{array}{*{20}c}
   {\left[ {\begin{array}{*{20}c}
   {P(\mu )} & 0  \\
   {\gamma _* P'(\mu )} & {P(\mu )}  \\
\end{array}} \right]\left[ {\begin{array}{*{20}c}
   {v_1 (\gamma _* )}  \\
   {v_2 (\gamma _* )}  \\
\end{array}} \right] = s_* \left[ {\begin{array}{*{20}c}
   {u_1 (\gamma _* )}  \\
   {u_2 (\gamma _* )}  \\
\end{array}} \right],}  \\
   {\left[ {\begin{array}{*{20}c}
   {P(\mu )} & 0  \\
   {\gamma _* P'(\mu )} & {P(\mu )}  \\
\end{array}} \right]\left[ {\begin{array}{*{20}c}
   {\hat v_1 (\gamma _* )}  \\
   {\hat v_2 (\gamma _* )}  \\
\end{array}} \right] = s_* \left[ {\begin{array}{*{20}c}
   {\hat u_1 (\gamma _* )}  \\
   {\hat u_2 (\gamma _* )}  \\
\end{array}} \right],}  \\
\end{array}} \right.
\]
or equivalently,
\begin{equation}  \label{ex1}
\left\{ {\begin{array}{*{20}l}
   {P(\mu )v_1 (\gamma _* ) = s_* u_1 (\gamma _* ),}  \\
   {\gamma _* P'(\mu )v_1 (\gamma _* ) + P(\mu )v_2 (\gamma _* ) = s_* u_2 (\gamma _* ),}  \\
   {P(\mu )\hat v_1 (\gamma _* ) = s_* \hat u_1 (\gamma _* ),}  \\
   {\gamma _* P'(\mu )\hat v_1 (\gamma _* ) + P(\mu )\hat v_2 (\gamma _* ) = s_* \hat u_2 (\gamma _* ),}
\end{array}} \right.
\end{equation}
and
\[
\left\{ {\begin{array}{*{20}c}
   {\left[ {\begin{array}{*{20}c}
   {u_1 (\gamma _* )^* } & {u_2 (\gamma _* )^* }  \\
\end{array}} \right]\left[ {\begin{array}{*{20}c}
   {P(\mu )} & 0  \\
   {\gamma _* P'(\mu )} & {P(\mu )}  \\
\end{array}} \right] = s_* \left[ {\begin{array}{*{20}c}
   {v_1 (\gamma _* )^* } & {v_2 (\gamma _* )^* }  \\
\end{array}} \right],}  \\
   {\left[ {\begin{array}{*{20}c}
   {\hat u_1 (\gamma _* )^* } & {\hat u_2 (\gamma _* )^* }  \\
\end{array}} \right]\left[ {\begin{array}{*{20}c}
   {P(\mu )} & 0  \\
   {\gamma _* P'(\mu )} & {P(\mu )}  \\
\end{array}} \right] = s_* \left[ {\begin{array}{*{20}c}
   {\hat v_1 (\gamma _* )^* } & {\hat v_2 (\gamma _* )^* }  \\
\end{array}} \right],}  \\
\end{array}} \right.
\]
or equivalently,
\begin{equation}  \label{ex2}
\left\{ {\begin{array}{*{20}l}
   {u_1 (\gamma _* )^* P(\mu ) + \gamma _* u_2 (\gamma _* )^* P'(\mu ) = s_* v_1 (\gamma _* )^* ,}  \\
   {u_2 (\gamma _* )^* P(\mu ) = s_* v_2 (\gamma _* )^* ,}  \\
   {\hat u_1 (\gamma _* )^* P(\mu ) + \gamma _* \hat u_2 (\gamma _* )^* P'(\mu ) = s_* \hat v_1 (\gamma _* )^* ,} \\
   {\hat u_2 (\gamma _* )^* P(\mu ) = s_* \hat v_2 (\gamma _* )^* . }
\end{array}} \right.
\end{equation}

By multiplying the fourth equation in (\ref{ex1}) by $u_2 (\gamma
_* )^*$ from the left, and the second equation of (\ref{ex2}) by
$\hat v_2 (\gamma _* )$ from the right, we obtain
\begin{equation}  \label{ex3}
\gamma _* u_2 (\gamma _* )^* P'(\mu )\hat v_1 (\gamma _* ) + u_2
(\gamma _* )^* P(\mu )\hat v_2 (\gamma _* ) = s_* u_2 (\gamma _*
)^* \hat u_2 (\gamma _* ),
\end{equation}
and
\begin{equation}  \label{ex4}
u_2 (\gamma _* )^* P(\mu )\hat v_2 (\gamma _* ) = s_* v_2 (\gamma
_* )^* \hat v_2 (\gamma _* ),
\end{equation}
respectively. As a consequence,
\begin{equation}  \label{ex5}
\gamma _* u_2 (\gamma _* )^* P'(\mu )\hat v_1 (\gamma _* ) = s_*
\left( {u_2 (\gamma _* )^* \hat u_2 (\gamma _* ) - v_2 (\gamma _*
)^* \hat v_2 (\gamma _* )} \right).
\end{equation}
Performing similar calculations, one can verify that
\begin{equation}  \label{ex6}
\gamma _* \hat u_2 (\gamma _* )^* P'(\mu )v_1 (\gamma _* ) = s_*
\left( {\hat u_2 (\gamma _* )^* u_2 (\gamma _* ) - \hat v_2
(\gamma _* )^* v_2 (\gamma _* )} \right).
\end{equation}
Clearly, equations (\ref{ex5}) and (\ref{ex6}) imply that the
non-diagonal entries of matrix $M$ are complex conjugate.
\end{proof}

By Lemma \ref{dotaee}\,(1), equation (\ref{alphabeta}) has always
a nontrivial (i.e., nonzero) solution, and hence, the hermitian
matrix $M$ in (\ref{matrixM}) cannot be (positive or negative)
definite. In our numerical example, $M$ has a negative and a
positive diagonal entries (namely, $- 2.6396$ and $4.1089$), and
thus, it is an indefinite hermitian matrix.

To derive an explicit solution of (\ref{alphabeta}), suppose that
$\eta_1,\eta_2 \in \mathbb{C}$ are the (real) eigenvalues of
matrix $M$, with $\eta_1 > 0 > \eta_2$, and let $w_1,w_2 \in
\mathbb{C}^2$ be unit eigenvectors of $M$ corresponding to
$\eta_1$ and $\eta_2$, respectively. Then, it is straightforward
to see (keeping in mind the orthogonality of the eigenvectors)
that the unit vector
\[
    \left[ {\begin{array}{*{20}{c}} \alpha \\  \beta  \end{array}} \right] =
    \sqrt{\frac{\left|\eta_2\right|}{\left|\eta_1\right|+\left|\eta_2\right|}}\,w_1
    + \sqrt{\frac{\left|\eta_1\right|}{\left|\eta_1\right|+\left|\eta_2\right|}}\,w_2
\]
satisfies
\[
  \left[ {\begin{array}{*{20}{c}}
  \overline{\alpha}& \overline{\beta}
  \end{array}} \right] \,M\,
  \left[ {\begin{array}{*{20}{c}}
  \alpha \\
  \beta
  \end{array}} \right]
  = \frac{\left|\eta_1\right| \eta_2}{\left|\eta_1\right|+\left|\eta_2\right|} +
    \frac{\left|\eta_2\right|\eta_1}{\left|\eta_1\right|+\left|\eta_2\right|}
  = 0 .
\]

Finally, in order to verify the validity of this refinement, we
return again to the normal matrix polynomial $P(\lambda)$ in
(\ref{pnormal}), and by applying the above methodology, we obtain
$\alpha  = 0.6254$ and $\beta  = 0.7803$. Consequently, the
desired vectors in (\ref{uvnew}) are (approximately)
\[
   \left[ {\begin{array}{*{20}c}
   {\tilde{u}_1 (\gamma_* )}  \\
   {\tilde{u}_2 (\gamma_* )}  \end{array}} \right]
   =  \left[  \begin{array}{c}
    0  \\  0.6560  \\  -0.2640  \\  0  \\   0.4226  \\  0.5669
    \end{array} \right]
   \;\;\; \mbox{and} \;\;\;
   \left[ {\begin{array}{*{20}c}
   {\tilde{v}_1 (\gamma_* )}  \\
   {\tilde{v}_2 (\gamma_* )}  \end{array}} \right]
   =  \left[  \begin{array}{c}
    0  \\   0.4226  \\ - 0.5669  \\  0  \\  0.6560  \\    0.2640
   \end{array} \right] .
\]
In particular, it holds that
\[
   \tilde{u}_2^*({\gamma_*}) P'(-4) \tilde{v}_1({\gamma _*}) = -4.4409 \cdot 10^{-16} ,
\]
and for the $n\times 2$ matrices
$\tilde{U}(\gamma_* ) = \left [ \tilde{u}_1 (\gamma_*) ~ \tilde{u}_2 (\gamma_*) \right]$
and $\tilde{V} (\gamma_*) = \left[ \tilde{v}_1(\gamma_*) ~ \tilde{v}_2 (\gamma_*) \right]$,
we have
\[
   \left\| \tilde{U}^* ({\gamma_*}) \tilde{U} ({\gamma_*}) -
   \tilde{V}^* ({\gamma _*}) \tilde{V} ({\gamma _*}) \right\|_2  = 1.1383 \cdot 10^{-6} .
\]
Thus, Lemma \ref{dotaee} is verified.

Moreover, using the matrices $\tilde{U} ({\gamma_*})$ and
$\tilde{V} ({\gamma_*})$, Theorem \ref{amultistar} yields the
upper bound $0.9465$ for the distance from $P(\lambda)$ to the set
of $3\times 3$ quadratic matrix polynomials that have $\mu=-4$ as
a multiple eigenvalue, and the perturbed matrix polynomial
   {\small
\[
 { \tilde{Q}_{{\gamma _*}}}(\lambda ) = \left[ {\begin{array}{*{20}{c}}
 1&0&0\\
 0&{0.0680}&{0.0152}\\
 0&{ - 0.1552}&{0.5986}
 \end{array}} \right] {\lambda^2}
 + \left[ {\begin{array}{*{20}{c}}
 { - 3}&0&0\\
 0&{ - 0.0680}&{ - 0.0152}\\
 0&{0.1552}&{3.4014}
 \end{array}} \right]\lambda + \left[ {\begin{array}{*{20}{c}}
2&0&0\\
0&{ - 0.9320}&{0.0152}\\
0&{ - 0.1552}&{1.5986}
\end{array}} \right]
\]
   }
that lies on the boundary of $\,\mathcal{B} \left( P , 0.9465 , \textup{w} \right)$
and has spectrum
\[
   \sigma \left( \tilde{Q}_{{\gamma _*}} (\lambda) \right) =
   \left\{ 1 ,\, 2 ,\, 4.1982 ,\, - 0.5140 ,\, - 4.0000 +  \textup{i}\,0.0031 ,\, - 4.0000  -  \textup{i}\,0.0031 \right\} .
\]
In addition, the lower bound $0.4031$ of the distance is given by
Theorem 11 in \cite{papa}. (All computations were performed in
Matlab with $16$ significant digits.)

\subsection{The case of multiplicity greater than 2}

Suppose that $\gamma_* > 0$, and the multiplicity of the singular
value $s_* > 0$ is $r \geq 3$. For weakly normal matrix polynomials,
this means that the graph of the increasing function $s_{n,1}(\gamma)$
intersects the graphs of more than one of the $n-1$ decreasing
functions $s_{1,2}(\gamma), s_{2,2}(\gamma) , \dots , s_{n-1,2}(\gamma)$,
at the point $(\gamma_*,s_*)$.

Let also
\[
  \left[ {\begin{array}{*{20}c} u_1^{(1)} (\gamma_*) \\ u_2^{(1)} (\gamma_*) \end{array}}
  \right],\,
  \left[ {\begin{array}{*{20}c} u_1^{(2)} (\gamma_*) \\ u_2^{(2)} (\gamma_*) \end{array}}
  \right],\, \dots ,\,
  \left[ {\begin{array}{*{20}c} u_1^{(r)} (\gamma_*) \\ u_2^{(r)} (\gamma_*) \end{array}} \right]
\]
and
\[
  \left[ {\begin{array}{*{20}c} v_1^{(1)} (\gamma_*) \\ v_2^{(1)}
(\gamma_*) \end{array}}
  \right],\,
  \left[ {\begin{array}{*{20}c} v_1^{(2)} (\gamma_*) \\ v_2^{(2)} (\gamma_*) \end{array}}
  \right],\, \dots ,\,
  \left[ {\begin{array}{*{20}c} v_1^{(r)} (\gamma_*) \\ v_2^{(r)} (\gamma_*) \end{array}} \right]
\]
be orthonormal bases of the left and right singular subspaces of
$F[P(\mu);\gamma_*]$ corresponding to $s_*$, respectively. Then,
we are looking for a pair of unit vectors
\begin{equation}  \label{uvnew2}
   \left[ {\begin{array}{*{20}c}
   {\tilde{u}_1 (\gamma_* )}  \\
   {\tilde{u}_2 (\gamma_* )}  \end{array}} \right]
   = \sum_{j=1}^r \alpha_j \left[ {\begin{array}{*{20}c}
    u_1^{(j)} (\gamma_*)  \\ u_2^{(j)} (\gamma_*)  \end{array}} \right]
   ,\;\;
   \left[ {\begin{array}{*{20}c}
   \tilde{v}_1 (\gamma_*) \\
   \tilde{v}_2 (\gamma_*) \end{array}} \right]
   = \sum_{j=1}^r \alpha_j  \left[ {\begin{array}{*{20}c}
   v_1^{(j)} (\gamma_*)  \\ v_2^{(j)} (\gamma_*) \end{array}} \right]
\end{equation}
such that
\begin{equation}\label{upvnew2}
  \tilde{u}_2({\gamma_*})^*P'(\mu)\tilde{v}_1({\gamma_*}) = 0 ,
\end{equation}
where the scalars $\alpha_1, \alpha_2, \dots, \alpha_r \in
\mathbb{C}$ satisfy $\left| \alpha_1 \right|^2 + \left| \alpha_2
\right|^2 + \cdots + \left| \alpha_r \right|^2 = 1$.

Following the arguments of the methodology described in the previous subsection,
we can compute the desired vectors in (\ref{uvnew2}) that satisfy
(\ref{upvnew2}). In particular, we need to find a solution of the
equation
\begin{equation} \label{comb}
    \left[ {\begin{array}{*{20}{cccc}}
    \overline{\alpha}_1 & \overline{\alpha}_2 & \cdots & \overline{\alpha}_r
    \end{array}} \right] \, M_r \,
    \left[ {\begin{array}{*{20}{c}}
    \alpha_1   \\
    \alpha_2   \\
    \vdots     \\
    \alpha_r \end{array}} \right]  =  0,
\end{equation}
where the $r\times r$ matrix
\[
 M_r = \left[ {\begin{array}{*{20}{cccc}}
 u_2^{(1)}({\gamma_*})^*P'(\mu)v_1^{(1)}({\gamma_*}) & u_2^{(1)}({\gamma_*})^*P'(\mu)v_1^{(2)}({\gamma_*}) & \cdots & u_2^{(1)}({\gamma_*})^*P'(\mu)v_1^{(r)}({\gamma_*})  \\
 u_2^{(2)}({\gamma_*})^*P'(\mu)v_1^{(1)}({\gamma_*}) & u_2^{(2)}({\gamma_*})^*P'(\mu)v_1^{(2)}({\gamma_*}) & \cdots & u_2^{(2)}({\gamma_*})^*P'(\mu)v_1^{(r)}({\gamma_*})  \\
          \vdots     &    \vdots     &     \ddots     &      \vdots   \\
 u_2^{(r)}({\gamma_*})^*P'(\mu)v_1^{(1)}({\gamma_*}) & u_2^{(r)}({\gamma_*})^*P'(\mu)v_1^{(2)}({\gamma_*}) & \cdots & u_2^{(r)}({\gamma_*})^*P'(\mu)v_1^{(r)}({\gamma_*})  \\
 \end{array}  } \right]
\]
is hermitian and not definite. Considering a unit eigenvector
$w_{\max} \in \mathbb{C}^r$ of $M_r$ corresponding to the maximum
eigenvalue $\eta_{\max} > 0$ of $M_r$ and an eigenvector $w_{\min}
\in \mathbb{C}^r$ corresponding to the minimum eigenvalue
$\eta_{\min} < 0$ of $M_r$, it is straightforward to verify that
the unit vector
\[
 \sqrt{\frac{\left|\eta_{\min}\right|}{\left|\eta_{\max}\right|+\left|\eta_{\min}\right|}}\,w_{\max}
 +
 \sqrt{\frac{\left|\eta_{\max}\right|}{\left|\eta_{\max}\right|+\left|\eta_{\min}\right|}}\,w_{\min}
\]
satisfies (\ref{comb}).


\end{document}